\documentclass{amsart}

\usepackage{amssymb,amsbsy,amsthm,amsmath,graphicx,epsfig}

\newtheorem{thm}{Theorem}[section]
\newtheorem*{thm*}{Theorem}

\newtheorem{prop}[thm]{Proposition}

\newtheorem{dfn}[thm]{Definition}

\theoremstyle{remark}
\newtheorem*{ex}{Example}
\newtheorem*{remark}{Remark}

\newcommand{\T}{\mathbb{T}}
\newcommand{\D}{\mathbb{D}}

\newcommand{\N}{\mathbb{N}}

\newcommand{\C}{\mathbb{C}}

\newcommand{\veps}{\varepsilon}

\newcommand{\lin}{\mathrm{lin}}
\newcommand{\aveN}{\frac{1}{N}\sum_{n=1}^N}
\newcommand{\limaveN}{\lim_{N\to\infty} \aveN}

\renewcommand{\hat}[1]{\widehat{#1}}

\newcommand{\la}{\langle}
\newcommand{\ra}{\rangle}

\title{Linear sequences and weighted ergodic theorems}
\author{Tanja Eisner}
\address
{Korteweg-de Vries Institute for Mathematics\\
University of Amsterdam\\
P.O.\ Box 94248\\
1090 GE Amsterdam\\
The Netherlands}
\email{t.eisner@uva.nl}


\begin{document}

\maketitle

\begin{abstract}
We present a simple way to produce good weights for several types of ergodic theorem including the Wiener-Wintner type multiple return time theorem and the multiple polynomial ergodic theorem. These weights are deterministic and come from orbits of certain bounded linear operators on Banach spaces. This extends the known results for nilsequences and return time sequences of the form $(g(S^ny))$ for a measure preserving system $(Y,S)$ and $g\in L^\infty(Y)$, avoiding in the latter case the problem of finding the full measure set of appropriate points $y$. 
\end{abstract}

\section{Introduction}

The classical mean and pointwise ergodic theorems due to von Neumann and Birkhoff, respectively, take their origin in questions from statistical physics and found applications in quite different areas of mathematics such as number theory, stochastics and harmonic analysis. Over the years, they were extended and ge\-ne\-ralised in many ways. For example, to multiple ergodic theorems, see e.g.~Furstenberg \cite{furst-book}, Bergelson, Leibman, Lesigne \cite{BLL}, Host, Kra \cite{host-kra}, Ziegler \cite{ziegler}, Tao \cite{tao}, to the Wiener-Wintner theorem, see e.g.~Assani \cite{assani-book}, Lesigne \cite{lesigneWW}, Frantzikinakis \cite{frantWW}, Host, Kra \cite{HK-unif-norms}, Eisner, Zorin-Kranich \cite{EZ}, to the return time theorem and its generalisations, see e.g.~Bourgain, Furstenberg, Katznelson, Ornstein \cite{BFKO},  Demeter, Lacey, Tao, Thiele \cite{DLTT}, Rudolph \cite{rudolphRTT}, Assani, Presser \cite{assani/presserRTT, assani/presser}, Zorin-Kranich \cite{pavel-RTT-mult}, and to further weighted, modulated and subsequential ergodic theorems, see e.g.~Berend, Lin, Rosenblatt, Tempelman \cite{berend/lin/rosenblatt/tempelman}, Below, Losert \cite{below/losert}, Bourgain \cite{bourgain89, bourgain88}, Wierdl \cite{wierdl}. 

The return time theorem due to Bourgain, solving a quite long standing open problem, is a classical example of a weighted pointwise ergodic theorem. It states that for every 
measure preserving system $(Y,\mu,S)$ and $g\in L^\infty(Y,\mu)$, the sequence $(g(S^ny))$ is for $\mu$-almost every $y$ a good weight for the pointwise ergodic theorem. This means that for every other 
system $(Y_{1},\mu_{1},S_{1})$ and every $g_{1}\in L^\infty(Y_{1},\mu_{1})$, the weighted ergodic averages 
$$
  \aveN g(S^ny) g_{1}(S_{1}^ny_{1})
$$
converge almost everywhere in $y_{1}$. The proof due to Bourgain, Furstenberg, Katznelson, Ornstein \cite{BFKO}, see also Lesigne, Mauduit, Moss\'e \cite{LMM} and Zorin-Kranich \cite{pavel-RTT}, is descriptive and gives conditions on $y$ to produce a good weight.  
However, these conditions can be quite difficult to check in a concrete situation. Later, Rudolph \cite{rudolphRTT}, see also Assani, Presser \cite{assani/presserRTT} and Zorin-Kranich \cite{pavel-RTT-mult}, gave a generalisation of the return time theorem and showed that (in the above notation) the sequence $(g(S^ny))$ is for almost every $y$ a universally good weight for multiple ergodic averages, see Definition \ref{def:universal} below. However, the conditions on the point $y$ did not become easier to check. 

The most general class of systems for which the convergence in the multiple return time theorem is known to hold \emph{everywhere}, hence leading to good weights which are easy to construct, are nilsystems, i.e., systems of the form $Y=G/\Gamma$ for a nilpotent Lie group $G$, a discrete cocompact subgroup $\Gamma$, the Haar measure $\mu$ on $G/\Gamma$ and the rotation $S$ by some element of $G$. For such a system $(Y,\mu,S)$, $g\in C(Y)$ and $y\in Y$, the sequence $(g(S^ny))$ is called a basic nilsequence. A \emph{nilsequence} is a uniform limit of basic nilsequences of the same step, or, equivalently, a sequence of the form $(g(S^ny))$ for an inverse limit $Y$ of nilsystems of the same step, $y\in Y$, a rotation $S$ on $Y$ and $g\in C(Y)$, see Host, Maass \cite{host/maass}. Indeed, recently Zorin-Kranich \cite{pavel-RTT-mult} proved the Wiener--Wintner type return time theorem for nilsequences showing universal convergence of averages 
\begin{equation}\label{eq:weighted-mult-RTT}
\aveN a_n g_1(S_1^ny_1) \cdots g_k(S_k^ny_k) 
\end{equation}
for every $k\in \N$ and every nilsequence $(a_n)$, where the universal sets of convergence do not depend on $(a_n)$. This generalised an earlier result by Assani, Lesigne, Rudolph \cite{ALR} for sequences of the form $(\lambda^n)$, $\lambda\in \T$, and $k=2$. 

In this paper we search for good weights for ergodic theorems using a functional analytic perspective and produce deterministic good weights. We first introduce sequences of the form $(\la T^nx,x'\ra)$, which we call \emph{linear sequences} if $x$ is in a Banach space $X$, $x'\in X'$ and $T$ is a linear operator on $X$ with relatively weakly compact orbits, see Section \ref{section:structure} below. Using a structure result for linear sequences, we show that they are good weights for the multiple polynomial ergodic theorem (Section \ref{section:mult}) and for the Wiener--Wintner type multiple return time theorem discussed above (Section \ref{section:ww-return}). 
In the last section we present a counterexample showing that the assumption on the operators cannot be dropped even for positive isometries on Banach lattices and the mean ergodic theorem.

We finally remark that all results in this note hold if we replace linear sequences by a larger class of ``asymptotic nilsequences'', i.e., for sequences $(a_n)$ of the form $a_n=b_n+c_n$, where $(b_n)$ is a nilsequence and $(c_n)$ is a bounded sequence satisfying $\limaveN |c_n|=0$ (cf.~Theorem \ref{thm:structure}). 
Examples of asymptotic nilsequences (of step $\geq 2$ in general) are \emph{multiple polynomial correlation sequences} $(a_{n})$ of the form 
$$
  a_{n}=\int_{Y} S^{p_{1}(n)}g_{1} \cdots S^{p_{k}(n)}g_{k}\, d\mu 
$$ 
for an ergodic invertible measure preserving system $(Y,\mu,S)$, $k\in\N$, $g_{j}\in L^{\infty}(Y,\mu)$ and polynomials $p_{j}$ with integer coefficients, $j=1,\ldots, k$. This follows from Leibman \cite[Theorem 3.1]{leibman10} and, in the case of linear polynomials, is due to Bergelson, Host, Kra \cite[Theorem 1.9]{berg/host/kra}. Thus multiple polynomial correlation sequences provide another class of deterministic examples of good weights for the Wiener--Wintner type multiple return time theorem and the multiple polynomial ergodic theorem discussed in Sections \ref{section:ww-return} and \ref{section:mult}.

\textbf{Acknowledgement.} We thank Pavel Zorin-Kranich for helpful discussions. 

\section{Linear sequences and their structure}\label{section:structure}

A linear operator $T$ on a Banach space $X$ has \emph{relatively weakly compact orbits} if for every $x\in X$, the orbit $\{T^nx,\, n\in \N_0\}$ is relatively weakly compact in $X$. 

\begin{dfn}
We call a sequence $(a_n)\subset \C$ a \emph{linear sequence} if there exist a relatively weakly compact operator $T$ on a Banach space $X$ and $x\in X$, $x'\in X'$ such that $a_n=\la T^nx,x' \ra$ holds for every $n\in \N$.  
\end{dfn}
A large class of relatively weakly compact operators, leading to a large class of linear sequences, are power bounded operators on reflexive Banach spaces. Recall that an operator $T$ is called \emph{power bounded} if it satisfies $\sup_{n\in\N} \|T^n\|<\infty$. Another class of relatively weakly compact operators are power bounded positive operators on a Banach lattice $L^1(\mu)$ preserving the order interval generated by a strictly positive function, see e.g.~Schaefer \cite[Theorem II.5.10(f) and Proposition II.8.3]{schaefer-book}. See \cite[Section I.1]{eisner-book} and \cite[Section 16.1]{EFHN} for further discussion. 
\begin{remark}
By restricting to the closed linear invariant subspace $Y:=\overline{\lin}\{T^nx,\, n\in \N_0\}$ induced by the orbit and using the decomposition $X'=Y'\oplus Y'_0$ for $Y'_0:=\{x':\, x'|_Y=0\}$, it suffices to assume that only the relevant orbit $\{T^nx,\, n\in \N_0\}$ is relatively weakly compact in the definition of a linear sequence $(\la T^n x,x'\ra)$. Note that in this case $T$ is relatively weakly compact on $Y$ by a limiting argument, see  e.g.~\cite[Lemma I.1.6]{eisner-book}.
\end{remark} 

We obtain the following structure result for linear sequences as a direct consequence of an extended Jacobs--Glicksberg-deLeeuw decomposition for relatively weakly compact operators. 
\begin{thm}\label{thm:structure}
Every linear sequence is a sum of an almost periodic sequence and a (bounded) sequence $(c_n)$ satisfying $\limaveN |c_n|=0$.
\end{thm}
Recall that by the Koopman-von Neumann lemma, see e.g.~Petersen \cite[p.~65]{petersen-book}, for bounded sequences the condition $\limaveN |c_n|=0$ is equivalent to $\lim_{j\to\infty}c_{n_{j}}=0$ for some subsequence $\{n_{j}\}\subset\N$ with density $1$.
\begin{proof}
Let $T$ be a relatively weakly compact operator on a Banach space $X$.
By the Jacobs-Glicksberg-deLeeuw decomposition, see e.g.~\cite[Theorem II.4.8]{eisner-book}, $X=X_r\oplus X_s$, where 
$$
  X_r=\overline{\lin}\{x:\, Tx=\lambda x \text{ for some }  \lambda \in \T\},
$$ 
while every $x\in X_s$ satisfies $\lim_{N\to\infty}\aveN |\la T^nx,x'\ra|=0$ for every $x'\in X'$. 

Let $x\in X$, $x'\in X'$ and define the sequence $(a_n)$ by $a_n:=\la T^nx,x'\ra$. 
For $x\in X_s$ we have $\limaveN |a_n|=0$ by the above. 
If now $x$ is an eigenvector corresponding to an eigenvalue $\lambda\in \T$, then $a_n=\lambda^n\la x,x'\ra$. Therefore for every $x\in X_r$, the sequence $(a_n)$ is a uniform limit of finite linear combinations of sequences $(\lambda^n)$, $\lambda\in \T$, and is therefore almost periodic. The assertion follows. 
\end{proof}

\section{A Wiener-Wintner type result for the multiple return time theorem}\label{section:ww-return}

In this section we show that one can take linear sequences as weights in the multiple Wiener-Wiener type generalisation of the return time theorem due to Zorin-Kranich \cite{pavel-RTT-mult} and Assani, Lesigne, Rudolph \cite{ALR} discussed in the introduction. 

First we recall the definition of a property satisfied universally. 

\begin{dfn}\label{def:universal}
Let $k\in \N$ and $P$ be a pointwise property for $k$ 
measure preserving dynamical systems. We say that a property $P$ is satisfied \emph{universally almost everywhere} if for every system $(Y_1,\mu_1, S_1)$ and every $g_1\in L^\infty(Y_1,\mu_1)$ there is a set $Y'_1\subset Y_1$ of full measure such that for every $y_1\in Y'_1$ and every system $(Y_2, \mu_2, S_2)$ ... for every system $(Y_k,\mu_k,S_k)$ and $g_k\in L^\infty (Y_k,\mu_k)$ there is a set $Y'_k\subset Y_k$ of full measure such that for every $y_{k} \in Y'_{k}$ the property $P$ holds.   
\end{dfn}

We show the following linear version of the Wiener-Wintner type multiple return time theorem.

\begin{thm} 
For every $k\in\N$, the weighted averages (\ref{eq:weighted-mult-RTT}) converge universally almost everywhere for every linear sequence $(a_n)$, where the universal sets $Y'_j$, $j=1,\ldots,k$, of full measure are independent of $(a_n)$. 
\end{thm}
\begin{proof}
By Theorem \ref{thm:structure}, we can show the assertion for almost periodic sequences and for $(a_n)$ satisfying $\limaveN |a_n|=0$ separately. For sequences from the second class, the assertion follows from the estimate 
$$
\left| \aveN a_n g_1(S_1^ny_1) \cdots g_k(S_k^ny_k) \right| \leq  \|g_1\|_\infty\cdots \|g_k\|_\infty  \aveN |a_n| 
$$
with a clear choice of $Y'_1,\ldots,Y'_k$.

Universal convergence for almost periodic sequences is a consequence of Zorin-Kranich's result \cite[Theorem 1.3]{pavel-RTT-mult} which shows the assertion for the larger class of nilsequences. 
\end{proof}

\section{Weighted multiple polynomial ergodic theorem}\label{section:mult}
 
Using the Host-Kra Wiener-Wintner type result for nilsequences and extending their result for linear polynomials from \cite{HK-unif-norms}, Chu \cite{chu} showed the following (see also \cite{EZ} for a slightly different proof). 
Let $(Y,\mu,S)$ be a
system  and $g\in L^\infty(Y,\mu)$. Then for almost every $y\in Y$, the sequence $(g(S^ny))$ is a \emph{good weight for the multiple polynomial ergodic theorem}, i.e., 
for the sequence of weights $(a_n)$ given by $a_{n}:=g(S^ny)$ and  for every $k\in\N$, 
the weighted multiple polynomial averages 
\begin{equation}\label{eq:weighted-multiple-norm-ave}
\aveN a_{n} S_1^{p_1(n)}g_1 \cdots S_1^{p_k(n)}g_k
\end{equation}
converge in $L^2$ for every system $(Y_1,\mu_1,S_1)$ with invertible $S_1$, every $g_1,\ldots,g_k\in L^\infty(Y_1,\mu_1)$ and every polynomials $p_1,\ldots,p_k$ with integer coefficients.    
 
The following result is a consequence of Chu \cite[Theorem 1.3]{chu}, the fact that the product of two nilsequences is again a nilsequence, and equidistribution theory for nilsystems, see e.g.~Parry \cite{parry} and Leibman \cite{leibman}.
\begin{thm}\label{thm:nilseq-good-polynomial}
Every nilsequence is a good weight for the multiple polynomial ergodic theorem. 
\end{thm}
%

This remains true when replacing a nilsequence by a linear sequence.

\begin{thm}
Every linear sequence is a good weight for the multiple polynomial ergodic theorem. 
\end{thm}
\begin{proof}
For an almost periodic sequence $(a_n)$, the averages (\ref{eq:weighted-multiple-norm-ave}) converge in $L^2$ by Theorem \ref{thm:nilseq-good-polynomial}. It is also clear that the averages (\ref{eq:weighted-multiple-norm-ave}) converge to $0$ in $L^\infty$ for every sequence $(a_n)$ satisfying $\limaveN |a_n|=0$. The assertion follows now from Theorem \ref{thm:structure}. 
\end{proof}

\section{A counterexample}\label{section:example}

The following example shows that if one does not assume relative weak compactness in the definition of linear sequences, each of the above results can fail dramatically even for positive isometries on Banach lattices. 

\begin{ex}
Let $X:=l^1$ and $T$ be the right shift operator, i.e., 
$$
  T(t_1,t_2,\ldots):=(0, t_1,t_2,\ldots).
$$ 

We first show that for every $\lambda\in \T$, $x=(t_{j})\in X$ and $x'=(s_{j})\in X'$ we have 
\begin{equation} \label{eq:ex-ces-lim-div}
 \lim_{N\to\infty}\left| \aveN \lambda^{n} \la T^nx,x' \ra  - \aveN \lambda^{n} s_n \sum_{j=1}^\infty \overline{\lambda}^{j} t_j \right| =0.
\end{equation}
Indeed, take $\veps >0$ and $J\in\N$ such that $\sum_{j=J+1}^\infty |t_j|<\veps$. Then for $N\in \N$ we have
\begin{eqnarray*}
&\ & \left| \aveN \lambda^{n}\la T^nx,x' \ra  -\aveN   \lambda^{n}  s_n \sum_{j=1}^\infty \overline{\lambda}^{j} t_j \right| \\
&\ &\quad   = \left| \aveN \lambda^{n} \sum_{j=1}^\infty t_j s_{n+j} - \aveN   \lambda^{n}  s_n \sum_{j=1}^\infty \overline{\lambda}^{j} t_j \right| \\
&\ &\quad  \leq \left| \aveN \lambda^{n} \sum_{j=1}^J t_j s_{n+j} - \aveN   \lambda^{n}  s_n \sum_{j=1}^J \overline{\lambda}^{j} t_j  \right| + 2\|x'\|_\infty \veps \\
&\ &\quad  = \left| \sum_{j=1}^J \overline{\lambda}^{j} t_j \frac{1}{N} \sum_{n=1+j}^{N+j} \lambda^{n} s_{n} - \sum_{j=1}^J \overline{\lambda}^{j} t_j  \aveN   \lambda^{n}  s_n  \right| + 2\|x'\|_\infty \veps \\
&\ &\quad  \leq  \frac{2J \|x\|_{1}\|x'\|_{\infty}}{N} + 2\|x'\|_\infty \veps.
\end{eqnarray*}
Choosing, i.e., $N>J\|x\|_{1}/\veps$ finishes the proof of (\ref{eq:ex-ces-lim-div}).

In particular, for $\lambda=1$ we see that the sequence $(\la T^{n}x,x'\ra)$ is Ces\`aro divergent for every $x=(t_{j})\in l^1$ with $\sum_{j=1}^\infty t_j\neq 0$ and for every $x'\in l^{\infty}$ which is Ces\`aro divergent. Note that the sets of such $x$ and $x'$ are open and dense in $l^{1}$ and $l^{\infty}$, respectively. (The assertion for $l^{1}$ is clear as well as the openess of the set of Ces\`aro divergent sequences in  $l^{\infty}$, and density follows from the fact that one can construct C\`esaro divergent sequences of arbitrarily small supremum norm.) Thus, for topologically very big sets of $x$ and $x'$ (with complements being nowhere dense), the sequence $(\la T^{n}x,x'\ra)$ is not a good weight for the mean ergodic theorem.

We further show that in fact for every $0\neq x\in l^{1}$ there is $\lambda\in \T$ so that for every $x'\in l^{\infty}$ from a dense open set, the sequence $(\lambda^{n} \la T^n x,x'\ra)$ is Ces\`aro divergent, implying that the sequence $(\la T^{n} x,x'\ra)$ is not a good weight for the mean ergodic theorem. 

Take $0\neq x=(t_{j})\in l^{1}$ and define the function $f$ on the unit disc $\D$ by $f(z):=\sum_{j=1}^{\infty}t_{j}z^{j}$. Then $f$ is a nonzero holomorphic function belonging to the Hardy space $H^{1}(\D)$. By Hardy space theory, see e.g.~Rosenblum, Rovnyak \cite[Theorem 4.25]{rosenblum/rovnyak-book}, there is a set $M\subset \T$ of positive Lebesgue measure such that for every $\lambda\in M$ we have
$$
 \lim_{r\to 1-} f(r\overline{\lambda}) = \sum_{j=1}^{\infty}\overline{\lambda}^{j} t_{j}\neq 0.
$$ 
For every such $\lambda$, by (\ref{eq:ex-ces-lim-div}) we see that the sequence $(\lambda^{n} \la T^{n} x,x'\ra)$ is Ces\`aro divergent for every $x'=(s_{j})\in l^{\infty}$ such that $ (\lambda^{j} s_j )$ is Ces\`aro divergent. The set of such $x'$ is open and dense in $l^{\infty}$ since it is the case for $\lambda=1$ and the multiplication operator $(s_{j})\mapsto (\lambda^{j}s_{j})$ is an invertible isometry. 
Thus for every $0\neq x\in l^1$ there is an open dense set of $x'\in l^\infty$ such that the sequence $(\la T^n x,x'\ra)$ fails to be a good weight for the mean ergodic theorem. 
\end{ex}

\end{document}